\documentclass{amsart}

\usepackage{amsmath}
\usepackage{amssymb}
\usepackage{amsthm}
\usepackage{graphicx}
\usepackage{wrapfig}

\newtheorem{theorem}{Theorem}[section]
\newtheorem{proposition}[theorem]{Proposition}
\newtheorem{corollary}[theorem]{Corollary}
\newtheorem{lemma}[theorem]{Lemma}
\theoremstyle{definition}

\newtheorem{remark}[theorem]{Remark}
\begin{document}

\title[The dealternating number and the alternation number of a closed 3-braid]
{The dealternating number and the alternation number\\ of a closed 3-braid}

\author{Tetsuya Abe}
\address{Department of Mathematics, Osaka City University, Sugimoto, Sumiyoshi-ku Osaka 558-8585, Japan} 
\email{t-abe@sci.osaka-cu.ac.jp}
\author{Kengo Kishimoto}
\address{Department of Mathematics, Osaka City University, Sugimoto, Sumiyoshi-ku Osaka 558-8585, Japan}
\email{k-kishi@sci.osaka-cu.ac.jp}
\address{Department of Mathematics, Osaka City University, Sugimoto, Sumiyoshi-ku Osaka 558-8585, Japan}
\email{jong@sci.osaka-cu.ac.jp}
\urladdr{http://www.ex.media.osaka-cu.ac.jp/~d07sa009/}


\maketitle

\begin{abstract}
We give an upper bound for the dealternating number of a closed 3-braid.
As applications, we determine the dealternating numbers,
the alternation numbers and the Turaev genera of some closed positive 3-braids. 
We also show that there exist infinitely many positive knots 
with any dealternating number (or any alternation number) and any braid index.
\end{abstract}

\section{Introduction}
A \textit{link} is a disjoint union of circles embedded in $S^3$, and a \textit{knot} is a link with one component.
Throughout this paper, all links are oriented.
A link diagram is \textit{$n$-almost alternating} if $n$ crossing changes in the diagram turn
the diagram into an alternating link diagram (cf.~\cite{adams-almost-alternating}).
A link is \textit{$n$-almost alternating} if it has an $n$-almost alternating diagram 
and no diagram that can be turned into an alternating diagram in fewer crossings.
We say that a link is \textit{almost alternating} if it is $1$-almost alternating.
It is known that all non-alternating knots of eleven or fewer crossings except $11_{n95}$ and $11_{n118}$
are almost alternating~\cite{adams-almost-alternating}, \cite{goda-2001-almost-alternating}. 
The first author showed that there exists $n$-almost alternating knot for any positive integer $n$~\cite{abe1}.
We say that a link $L$ has \textit{dealternating number} $n$ if it is $n$-almost alternating.
We denote by $\mathrm{dalt} (L)$ the dealternating number of a link $L$.
We note that all non-alternating Montesinos links and all semi-alternating links are almost alternating (see Appendix).

Kawauchi introduced the alternation number of a link~\cite{kaw1},
which measures how far it is from alternating links.
Let $L$ and $L'$ be links.
The \textit{Gordian distance} $d_G(L,L')$ between $L$ and $L'$~\cite{murakami1}
is the minimal number of crossing changes needed to deform a diagram of $L$ into that of $L'$,
where the minimum is taken over all diagrams of $L$.
The \textit{alternation number} $\mathrm{alt} (L)$ of a link $L$
is the minimal number of $d_G(L,L')$ among all alternating links $L'$.
By definition, we have $\mathrm{alt} (L)\leq \mathrm{dalt} (L)$ for any link $L$.
We note that the alternation number differs from the dealternating number.
For example, 
the Whitehead double of the trefoil is 2-almost alternating~\cite{adams-almost-alternating}, 
however it has alternation number one. 
The first author gave a lower bound for the alternation number of a knot,
and determined the torus knots with alternation number one~\cite{abe1}.
By using this lower bound, Kanenobu determined the alternation numbers of some torus knots~\cite{kan1}.

In this paper, we give an upper bound for the dealternating number of a closed 3-braid
(Theorem~\ref{thm:dalt-upbound-3braid}).
As an application, we determine the dealternating numbers and the alternation numbers of some
closed positive $3$-braid knots as follows. \vspace{2mm}\\
{\bf Theorem 3.1.~}
	Let $\beta$ be a $3$-braid of the form
	\[
	\Delta^{2n} \prod_{i=1}^{r} \sigma_1^{p_i} \sigma_2^{q_i} 
	\] 
	such that $\widehat{\beta}$ is a knot and $p_i,q_i\geq 2$ for $i=1,2,\dots ,r$. Then we have
	\[
	\mathrm{alt}(\widehat{\beta})=\mathrm{dalt}(\widehat{\beta})=n+r-1.
	\]

The \textit{braid index} $b(L)$ of a link $L$ is the minimal number of strands
of any braid whose closure is equivalent to $L$. We show that there exist 
infinitely many positive knots with any dealternating number (or any alternation number) and any braid index.
 \vspace{2mm}\\
{\bf Theorem 4.1.~}
	For any $m\in \mathbb{Z}_{\geq 0}$ and $n\geq 3$, there exist infinitely many positive knots $K$
	with $\mathrm{alt} (K)=\mathrm{dalt} (K)=m$ and $b(K)=n$.
\vspace{2mm}

To a knot diagram $D$, Turaev associated a closed orientable surface embedded in $S^3$, 
called the \textit{Turaev surface}.
The definition of the Turaev surface is given in Section~\ref{Note on the Turaev surface} 
(see also~\cite{cromwell-book}, \cite{das2-graphs-on-surfaces}, \cite{lowrance2007kfw}).
The \textit{Turaev genus $g_{T}(K)$}  of a knot $K$ was first defined in~\cite{das2-graphs-on-surfaces}
as the minimal number of the genera of the Turaev surfaces of diagrams of $K$. 
We determine the Turaev genera and the dealternating numbers of the $(3,3n+i)$-torus knots, where $i=1,2$.\vspace{2mm}\\
{\bf Theorem 5.9.~}
	 Let $T_{3,3n+i}$ be the $(3,3n+i)$-torus knot, where $i=1,2$. Then we have 
	\[
	g_{T}(T_{3,3n+i})  = \mathrm{dalt}(T_{3,3n+i}) =n.
	\]

This paper is constructed as follows:
In Section~\ref{An upper bound for the dealternating number of a closed braid}, 
we give an upper bound for the dealternating number of a closed 3-braid.
In Section~\ref{The alternation number and the dealternating number of a closed positive 3-braid}, 
we determine the dealternating numbers and the alternation numbers of
some closed positive 3-braid knots by using a lower bound introduced by the first author~\cite{abe1}. 
In Section~\ref{A relation between the alternateness},
we show that there exist 
infinitely many positive knots with any dealternating number (or alternation number) and any braid index. 
In Section~\ref{Note on the Turaev surface}, we determine the dealternating numbers and the Turaev genera of 
the $(3,3n+i)$-torus knots, where $i=1,2$. 
In Appendix, we observe that non-alternating Montesinos links and semi-alternating links are almost alternating.  	
This Appendix is authored by the first and the second authors and In Dae Jong. 
\section{An upper bound for the dealternating number of a closed $3$-braid}
\label{An upper bound for the dealternating number of a closed braid}

In this section, we give an upper bound for the dealternating number of a closed $3$-braid.
The \textit{$n$-braid group} $B_n$, $n \in \mathbb{Z}_{>0}$, is a group which has the following 
presentation. 
\begin{equation*}
\left< \sigma_1,  \sigma_2, \dots ,  \sigma_{n-1}  \left| 
\begin{array}{cc}
\sigma_t \sigma_s = \sigma_s \sigma_t & (|t-s| >1)\\ 
\sigma_t \sigma_s \sigma_t = \sigma_s \sigma_t \sigma_s & (|t-s| =1)
\end{array}
\right\rangle \right. .
\end{equation*}

Let $\beta$ be an $n$-braid of the form
\[
\prod_{j=1}^{m} \sigma_{i_j}^{a_j},
\]
where $i_j\ne i_{j+1}$ and $a_j$ is nonzero integer for $j=1,2,\dots ,m$.
We call $\sigma_{i_j}^{a_j}$ the \textit{syllables} of $\beta$, 
which represent an $a_j$-half twists on two strands in $\beta$. 
We denote by $\widehat{\beta}$ the closure of a braid $\beta$.
Set
\begin{eqnarray*}
	E_+(\beta) &=& \{ ~j~|~ \text{$i_j$ is even, and $a_j$ is positive} \} ,\\
	E_-(\beta) &=& \{ ~j~|~ \text{$i_j$ is even, and $a_j$ is negative} \} ,\\
	O_+(\beta) &=& \{ ~j~|~ \text{$i_j$ is odd, and $a_j$ is positive} \} ,\\
	O_-(\beta) &=& \{ ~j~|~ \text{$i_j$ is odd, and $a_j$ is negative} \} .
\end{eqnarray*}
We denote by $|X|$ the number of elements of a set $X$.

\begin{lemma} \label{lem:daltupbound}
	Let $\beta$ be an $n$-braid. Then we have
	\[
	\mathrm{dalt}(\widehat{\beta})\leq \min\{|E_+(\beta )|+|O_-(\beta )|, |E_-(\beta )|+|O_+(\beta )|\}.
	\]
\end{lemma}

\begin{proof}
Let $D$ be the closed braid diagram of $\widehat{\beta}$.
We may assume that $|E_+(\beta )|+|O_-(\beta )|\leq |E_-(\beta )|+|O_+(\beta )|$ since
we can prove in a similar way as in the case $|E_+(\beta )|+|O_-(\beta )|\geq |E_-(\beta )|+|O_+(\beta )|$.
Let $T_j$ be a 2-tangle diagram in $D$, which is represented by a syllable $\sigma_{i_j}^{a_j}$ for $j=1,2,\dots ,m$.
For all $j \in E_+(\beta)\cup O_-(\beta)$, we deform $T_j$ by the Reidemeister moves,
and we apply a crossing change at the crossing $c$ as in Figure~\ref{fig:Figure1}.
\begin{figure}[htbp]
	\begin{center}
		\includegraphics[scale=0.6]{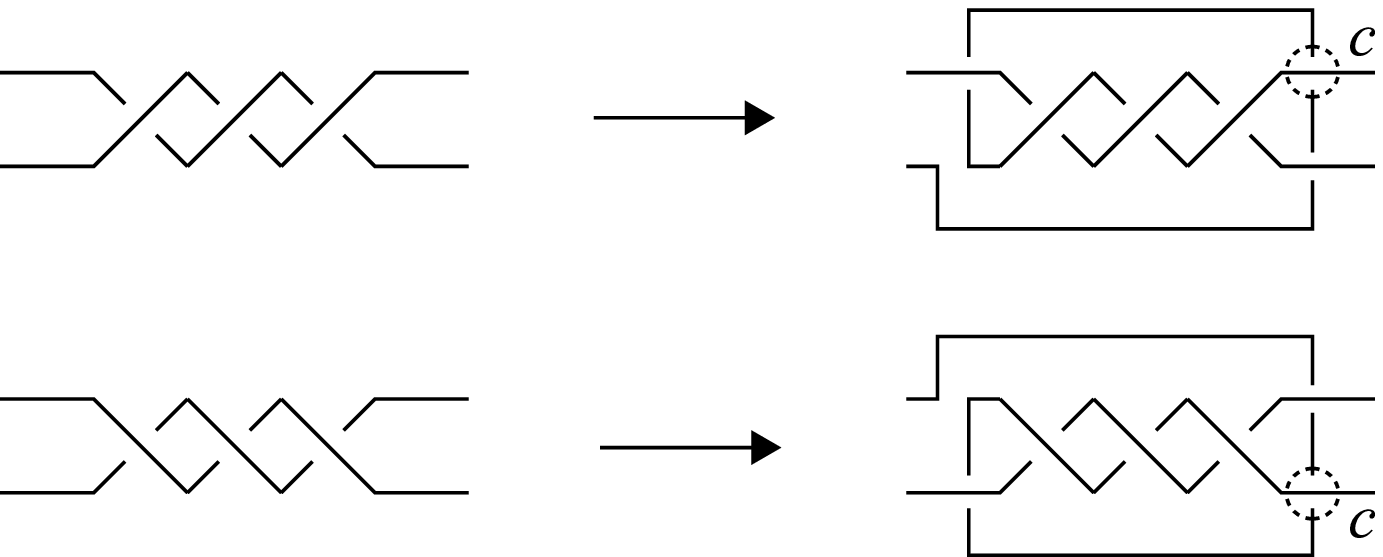}
	\end{center}
	\caption{}
	\label{fig:Figure1}
\end{figure}

By this deformation, we obtain an alternating diagram. Thus we have 
\[
\mathrm{dalt}(\widehat{\beta }) \leq |E_+(\beta )|+|O_-(\beta )|.
\]
\end{proof}

An $n$-braid $\beta$ is \textit{positive} if $a_j$ is positive for all $j$.
We note that a positive $3$-braid is conjugate to a braid of the form
\[
\prod_{i=1}^{r} \sigma_1^{p_i} \sigma_2^{q_i},
\] 
where $p_i,~q_i\in \mathbb{Z}_{>0}$ for $i=1,2,\dots ,r$. 
\begin{lemma} \label{lem:daltupboundpositive3}
	Let $\beta$ be a positive $3$-braid of the form
	\[
	\prod_{i=1}^{r} \sigma_1^{p_i} \sigma_2^{q_i},
	\] 
	where $p_i,~q_i\in \mathbb{Z}_{>0}$ for $i=1,2,\dots ,r$. Then we have
	\[
	\mathrm{dalt}(\widehat{\beta})\leq r-1.
	\]
\end{lemma}

\begin{proof}
Let $D$ be the closed braid diagram of $\widehat{\beta}$.
We deform $D$ into an alternating diagram by the Reidemeister moves and
a crossing change at each crossing $c_i~(i=1,2,\dots ,r-1)$ as in Figure~\ref{fig:Figure2}.
\begin{figure}[htbp]
  \begin{center}
    \includegraphics[scale=0.8]{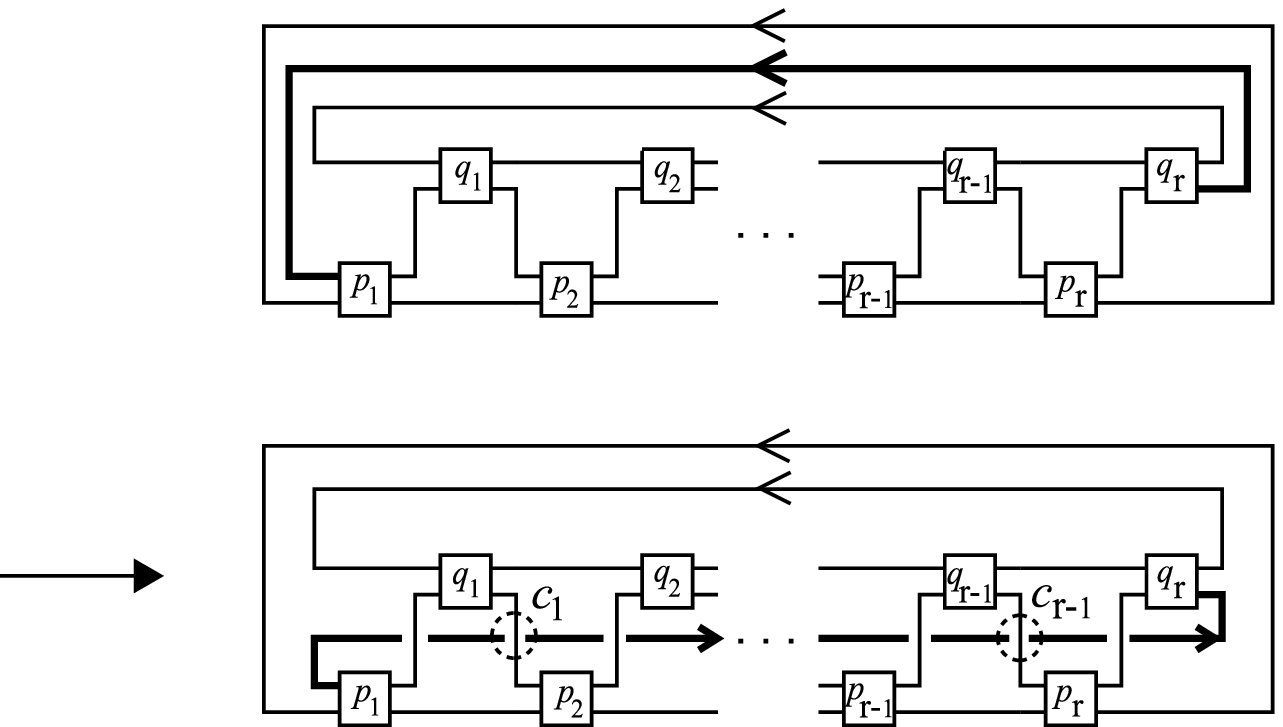}
  \end{center}
  \caption{}
  \label{fig:Figure2}
\end{figure}

\end{proof}

Murasugi showed the following.
\begin{lemma}[\cite{mur4}] \label{lem:3-braidclass} 
	Every $3$-braid is conjugate to an element of the following sets.
	\begin{eqnarray*}
	\Omega _0&=&\{\Delta^{2n}~|~n\in \mathbb{Z}\}, \\
	\Omega _1&=&\{\Delta^{2n}(\sigma_1 \sigma_2)~|~n\in \mathbb{Z}\},\\
	\Omega _2&=&\{\Delta^{2n}(\sigma_1 \sigma_2)^2~|~n\in \mathbb{Z}\},\\
	\Omega _3&=&\{\Delta^{2n+1}~|~n\in \mathbb{Z}\},\\
	\Omega _4&=&\{\Delta^{2n}\sigma_1^{-p}~|~n\in \mathbb{Z}, p\in \mathbb{Z}_{>0}\},\\
	\Omega _5&=&\{\Delta^{2n}\sigma_2^q~|~n\in \mathbb{Z}, q\in \mathbb{Z}_{>0}\},\\
	\Omega _6&=&\{\Delta^{2n}\prod_{i=1}^{r}\sigma_1^{-p_i} \sigma_2^{q_i}
                ~|~n\in \mathbb{Z}, r,p_i,q_i\in \mathbb{Z}_{>0}~(i=1,2,\dots ,r)\},
	\end{eqnarray*}
where $\Delta=\sigma_1 \sigma_2 \sigma_1$.
\end{lemma}

We denote by $\widehat{\Omega}_j$
$(j=0,\dots, 6)$ the set of the closures of braids in $\Omega _j$ $(j=0,\dots, 6)$.
\begin{remark}
The closure of the $3$-braid $\Delta^{2n}(\sigma_1 \sigma_2)^i \in \Omega _i$ is the $(3, 3n+i)$-torus link,
where $i=0, 1, 2$. The sets $\Omega _j$ $(j=0,\dots, 6)$ are mutually disjoint.
However, the sets $\widehat{\Omega}_j$ $(j=0,\dots, 6)$ are not mutually disjoint 
(see~\cite{birman1993-classification}, \cite{mur4}).
A closed $3$-braid knot is either the $(3, 3n+i)$-torus knot or a knot in $\widehat{\Omega }_6$.
\end{remark}

\begin{theorem}\label{thm:dalt-upbound-3braid}
	Let $\beta$ be a $3$-braid in $\Omega _i~(i=0,\dots ,6)$. With the notation of Lemma~\ref{lem:3-braidclass}, 
	we have
	\[
	\mathrm{dalt}(\widehat{\beta})\leq |n|.
	\]
	In particular, we have
	\begin{align*}
		(1) &~ \mathrm{dalt}(\widehat{\beta })\leq |n|-1& &\text{if } \beta \in \Omega _i ~(i=1,2),~\text{and}~n<0,\\
		(2) &~ \mathrm{dalt}(\widehat{\beta })\leq n-1  & &\text{if } \beta \in \Omega _4, ~n\geq 0,~\text{and} ~p=2n-1,2n,2n+1,\\
		(3) &~ \mathrm{dalt}(\widehat{\beta })\leq |n|-1& &\text{if } \beta \in \Omega _5, ~n\leq 0,~\text{and} ~p=2n-1,2n,2n+1.
	\end{align*}
\end{theorem}

\begin{proof}
{\bf Case 0:} $\beta \in \Omega _0$.
	We may assume that $n\geq 0$, since the mirror image of $\Delta^{2n}$ is $\Delta^{-2n}$.
	We can modify $\beta $ into $\sigma_1^{2n} \sigma_2 (\sigma_1^2 \sigma_2^2 )^{n-1}\sigma_1^2 \sigma_2$ 
	by braid relations and conjugations as follows.
	\begin{align*}
		\beta =\Delta^{2n} 
			 &=(\sigma_1 \sigma_2 \sigma_1 )^{2n}\\
			 &=(\sigma_1 \sigma_2 \sigma_1 \sigma_1 \sigma_2 \sigma_1 )^{n}\\				   
			 &=(\sigma_1^2 \sigma_2 \sigma_1^2 \sigma_2 )^{n}\\						   
			 &=\sigma_1^{2n} (\sigma_2 \sigma_1^2 \sigma_2 )^{n}\\
			 &=\sigma_1^{2n} \sigma_2 (\sigma_1^2 \sigma_2^2 )^{n-1}\sigma_1^2 \sigma_2 .
	\end{align*}
	We note that $\sigma_1 (\sigma_2 \sigma_1^2 \sigma_2)=(\sigma_2 \sigma_1^2 \sigma_2) \sigma_1$.
	By Lemma~\ref{lem:daltupboundpositive3}, we have 
	\[
	\mathrm{dalt}(\widehat{\beta})\leq n.
	\]
The rest of the proof is similar to the proof of Case 0. \\
{\bf Case 1.1:} $\beta \in \Omega _1$ and $n\geq0$.
	We modify $\beta $ into $\sigma_1^{2n+1} \sigma_2 (\sigma_1^2 \sigma_2^2 )^{n}$ 
	by the following equalities.
	\begin{align*}
		\beta =\Delta^{2n}(\sigma_1 \sigma_2 ) 
		 &=(\sigma_1 \sigma_2 \sigma_1 )^{2n} (\sigma_1 \sigma_2 )\\
		 &=\sigma_1^{2n} (\sigma_2 \sigma_1^2 \sigma_2 )^{n} \sigma_1 \sigma_2\\
		 &=\sigma_1^{2n+1} (\sigma_2 \sigma_1^2 \sigma_2 )^{n} \sigma_2\\
		 &=\sigma_1^{2n+1} \sigma_2 (\sigma_1^2 \sigma_2^2 )^{n} .
	\end{align*}
	By Lemma~\ref{lem:daltupboundpositive3}, we have 
	\[
	\mathrm{dalt}(\widehat{\beta})\leq n.
	\]
{\bf Case 1.2:} $\beta \in \Omega _1$ and $n<0$.
    We modify $\beta $ into $\sigma_1^{-2|n|-1} \sigma_2^{-1} (\sigma_1^2 \sigma_2^2 )^{-|n|+1}$ 
	by the following equalities.
	\begin{align*}
		\beta =\Delta^{-2|n|}(\sigma_1 \sigma_2 ) 
			 &=(\sigma_1 \sigma_2 \sigma_1 )^{-2|n|} (\sigma_1 \sigma_2 )\\
		     &=(\sigma_1 \sigma_2 \sigma_1 )^{-2|n|+2} (\sigma_1 \sigma_2 )^{-2}\\
			 &=\sigma_1^{-2|n|+2} (\sigma_2 \sigma_1^2 \sigma_2 )^{-|n|+1} (\sigma_1 \sigma_2)^{-2}\\
			 &=\sigma_1^{-2|n|-1} (\sigma_2 \sigma_1^2 \sigma_2 )^{-|n|+1}  \sigma_2^{-1}\\
			 &=\sigma_1^{-2|n|-1} \sigma_2^{-1} (\sigma_1^2 \sigma_2^2 )^{-|n|+1} .
	\end{align*}
	By Lemma~\ref{lem:daltupboundpositive3}, we have  
	\[
	\mathrm{dalt}(\widehat{\beta})\leq |n|-1.
	\]
{\bf Case 2.1:} $\beta \in \Omega _2$ and $n\geq 0$.
    We modify $\beta $ into $\sigma_1^{2n+3} \sigma_2 (\sigma_1^2 \sigma_2^2 )^{n}$ 
	by the following equalities.
	\begin{align*}
		\beta =\Delta^{2n} (\sigma_1 \sigma_2 )^2 
			 &=(\sigma_1 \sigma_2 \sigma_1 )^{2n} (\sigma_1 \sigma_2 )^2\\
			 &=\sigma_1^{2n} (\sigma_2 \sigma_1^2 \sigma_2 )^{n} (\sigma_1 \sigma_2 )^2 \\
			 &=\sigma_1^{2n} (\sigma_2 \sigma_1^2 \sigma_2 )^{n} \sigma_1^2 \sigma_2 \sigma_1 \\
			 &=\sigma_1^{2n+3} (\sigma_2 \sigma_1^2 \sigma_2 )^{n} \sigma_2 \\
			 &=\sigma_1^{2n+3} \sigma_2 (\sigma_1^2 \sigma_2^2 )^{n} .
	\end{align*}
	By Lemma~\ref{lem:daltupboundpositive3}, we have 
	\[
	\mathrm{dalt}(\widehat{\beta})\leq n.
	\]
{\bf Case 2.2:} $\beta \in \Omega _2$ and $n<0$.
    We modify $\beta $ into $\sigma_1^{-2|n|+1} \sigma_2^{-1} (\sigma_1^2 \sigma_2^2 )^{-|n|+1}$ 
	by the following equalities.
	\begin{align*}
		\beta =\Delta^{-2|n|} (\sigma_1 \sigma_2 )^2 
			 &=(\sigma_1 \sigma_2 \sigma_1 )^{-2|n|} (\sigma_1 \sigma_2 )^2\\
		     &=(\sigma_1 \sigma_2 \sigma_1 )^{-2|n|+2} (\sigma_1 \sigma_2 )^{-1}\\
			 &=\sigma_1^{-2|n|+2} (\sigma_2 \sigma_1^2 \sigma_2 )^{-|n|+1} (\sigma_1 \sigma_2)^{-1}\\
			 &=\sigma_1^{-2|n|+1} (\sigma_2 \sigma_1^2 \sigma_2 )^{-|n|+1}  \sigma_2^{-1}\\
			 &=\sigma_1^{-2|n|+1} \sigma_2^{-1} (\sigma_1^2 \sigma_2^2 )^{-|n|+1} .
	\end{align*}
	By Lemma~\ref{lem:daltupboundpositive3}, we have
	\[
	\mathrm{dalt}(\widehat{\beta})\leq |n|-1.
	\]
{\bf Case 3:} $\beta \in \Omega _3$. 
	We may assume that $n\geq 0$, since the mirror image of $\Delta^{2n+1}$ is $\Delta^{-(2n+1)}$.
	We modify $\beta $ into $\sigma_1^{2n+2} \sigma_2 (\sigma_1^2 \sigma_2^2 )^{n}$ 
	by the following equalities.
	\begin{align*}
	\beta =\Delta^{2n+1} &=(\sigma_1 \sigma_2 \sigma_1 )^{2n+1}\\
						 &=\sigma_1^{2n} (\sigma_2 \sigma_1^2 \sigma_2 )^{n} \sigma_1 \sigma_2 \sigma_1\\
						 &=\sigma_1^{2n+2} (\sigma_2 \sigma_1^2 \sigma_2 )^{n} \sigma_2\\
						 &=\sigma_1^{2n+2} \sigma_2 (\sigma_1^2 \sigma_2^2 )^{n} .
	\end{align*}
	By Lemma~\ref{lem:daltupboundpositive3}, we have
	\[
	\mathrm{dalt}(\widehat{\beta})\leq n .
	\]

Let $\beta $ be a $3$-braid in $\Omega _4$.
We may assume that $n\geq 0$, since the mirror image of $\Delta^{-2|n|} \sigma_1^{-p}$ is in $\Omega _5$.
We modify $\beta $ into $\sigma_1^{2n-p} \sigma_2 (\sigma_1^2 \sigma_2^2 )^{n-1}\sigma_1^2 \sigma_2$ 
by the following equalities.
    \begin{align*}
	\beta =\Delta^{2n} \sigma_1^{-p} 
		 &=(\sigma_1 \sigma_2 \sigma_1 )^{2n} \sigma_1^{-p} \\
		 &=\sigma_1^{2n-p} \sigma_2 (\sigma_1^2 \sigma_2^2 )^{n-1}\sigma_1^2 \sigma_2 .
	\end{align*}
{\bf Case 4.1:} $\beta \in \Omega _4$ and $p=2n$.
    We modify $\beta $ into $(\sigma_1^2 \sigma_2^2 )^n$ by the following equalities.
	\begin{align*}
	\beta &=\sigma_2 (\sigma_1^2 \sigma_2^2 )^{n-1}\sigma_1^2 \sigma_2\\
		  &=(\sigma_1^2 \sigma_2^2 )^n .
	\end{align*}
	By Lemma~\ref{lem:daltupboundpositive3}, we have
	\[
	\mathrm{dalt}(\widehat{\beta})\leq n-1.
	\]
{\bf Case 4.2:} $\beta \in \Omega _4$ and $p=2n-1$.
    We modify $\beta $ into $\sigma_1^2 \sigma_2^4 \sigma_1 \sigma_2^3 (\sigma_1^2 \sigma_2^2 )^{n-2}$
	by the following equalities.
	\begin{align*}
	\beta &=\sigma_1 \sigma_2 (\sigma_1^2 \sigma_2^2 )^{n-1}\sigma_1^2 \sigma_2\\
		  &=\sigma_1^2 \sigma_2 \sigma_1 \sigma_2 (\sigma_1^2 \sigma_2^2 )^{n-1}\\
		  &=\sigma_1^3 \sigma_2 \sigma_1^3 \sigma_2^2 (\sigma_1^2 \sigma_2^2 )^{n-2}\\
		  &=\sigma_1^2 \sigma_2^4 \sigma_1 \sigma_2^3 (\sigma_1^2 \sigma_2^2 )^{n-2} .
	\end{align*}
	By Lemma~\ref{lem:daltupboundpositive3}, we have
	\[
	\mathrm{dalt}(\widehat{\beta})\leq n-1 .
	\]
{\bf Case 4.3:} $\beta \in \Omega _4$ and $p=2n+1$.
	We deform $\widehat{\beta}$ into an $(n-1)$-almost alternating diagram
	by the Reidemeister moves as in Figure~\ref{fig:Figure3}. 
	\begin{figure}[htbp]
		\begin{center}
			\includegraphics[scale=0.6]{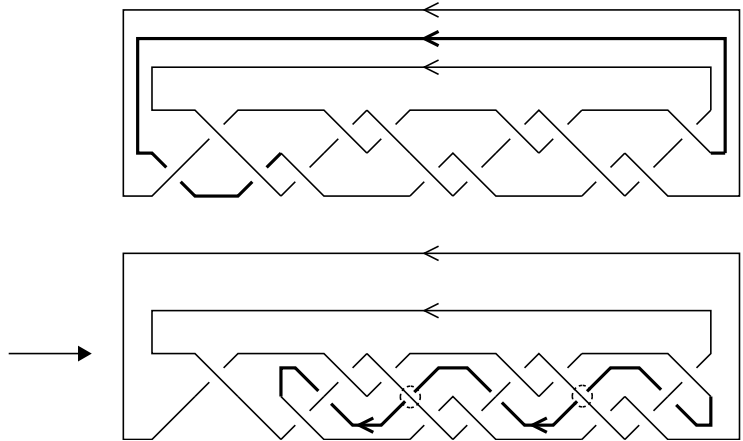}
		\end{center}
		\caption{The case for $n=3$}
		\label{fig:Figure3}
	\end{figure}
	
	Thus we have
	\[
	\mathrm{dalt}(\widehat{\beta})\leq n-1 .
	\]
{\bf Case 4.4:} $\beta \in \Omega _4$ and $p<2n-1$.
	By Lemma~\ref{lem:daltupboundpositive3}, we have
	\[
	\mathrm{dalt}(\widehat{\beta})\leq n .
	\]
{\bf Case 4.5:} $\beta \in \Omega _4$ and $p>2n+1$.
	By Lemma~\ref{lem:daltupbound}, we have
	\[
	\mathrm{dalt}(\widehat{\beta})\leq n .
	\]
{\bf Case 5:} $\beta \in \Omega _5$.
	We may assume that $n\geq 0$, since the closure of the mirror image of $\Delta^{-2|n|} \sigma_2^{q}$
	is in $\widehat{\Omega } _4$. 
	We modify $\beta $ into $\sigma_2^{2n+q} \sigma_1 (\sigma_2^2 \sigma_1^2 )^{n-1}\sigma_2^2 \sigma_1$
	by the following equalities.
	\begin{align*}
	\beta =\Delta^{2n} \sigma_2^{q} 
		 &=(\sigma_1 \sigma_2 \sigma_1 )^{2n} \sigma_2^{q} \\
		 &=\sigma_2^{2n+q} \sigma_1 (\sigma_2^2 \sigma_1^2 )^{n-1}\sigma_2^2 \sigma_1 .
	\end{align*}
	By Lemma~\ref{lem:daltupboundpositive3}, we have
	\[
	\mathrm{dalt}(\widehat{\beta})\leq n .
	\]
{\bf Case 6.1:} $\beta \in \Omega _6$ and $n\geq 0$.\\ 
	We modify $\beta $ into $\sigma_1 (\sigma_2^2 \sigma_1^2 )^{n-1} \sigma_2^2 \sigma_1^{-p_1+1} \sigma_2^{q_1}
	( \prod_{i=2}^{r-1}\sigma_1^{-p_i} \sigma_2^{q_i} ) \sigma_1^{-p_r} \sigma_2^{q_r+2n}$
	by the following equalities.
	\begin{align*}
	\beta &=(\sigma_1 \sigma_2 \sigma_1 )^{2n} \prod_{i=1}^{r}\sigma_1^{-p_i} \sigma_2^{q_i}\\
		  &=\sigma_2^{2n} (\sigma_1 \sigma_2^2 \sigma_1 )^{n} \prod_{i=1}^{r}\sigma_1^{-p_i} \sigma_2^{q_i}\\
		  &=\sigma_1 (\sigma_2^2 \sigma_1^2 )^{n-1} \sigma_2^2 \sigma_1^{-p_1+1} \sigma_2^{q_1}
			\Bigg( \prod_{i=2}^{r-1}\sigma_1^{-p_i} \sigma_2^{q_i} \Bigg) \sigma_1^{-p_r} \sigma_2^{q_r+2n} .
	\end{align*}
By Lemma~\ref{lem:daltupbound}, we have
\[
\mathrm{dalt}(\widehat{\beta})\leq n .
\]
{\bf Case 6.2:} $\beta \in \Omega _6$ and $n< 0$. \\
	We modify $\beta $ into $(\sigma_1^2 \sigma_2^2 )^{-|n|+1} \sigma_1^{-2} \sigma_2^{-1} \sigma_1^{-p_1-2|n|} \sigma_2^{q_1}
	( \prod_{i=2}^{r-1}\sigma_1^{-p_i} \sigma_2^{q_i} ) \sigma_1^{-p_r} \sigma_2^{q_r-1} $
	by the following equalities.
	\begin{align*}
	\beta &=(\sigma_1 \sigma_2 \sigma_1 )^{-2|n|} \prod_{i=1}^{r}\sigma_1^{-p_i} \sigma_2^{q_i}\\
		  &=(\sigma_2 \sigma_1^2 \sigma_2 )^{-|n|} \sigma_1^{-2|n|} \prod_{i=1}^{r}\sigma_1^{-p_i} \sigma_2^{q_i}\\
		  &=(\sigma_1^2 \sigma_2^2 )^{-|n|+1} \sigma_1^{-2} \sigma_2^{-1} \sigma_1^{-p_1-2|n|} \sigma_2^{q_1}
			\Bigg( \prod_{i=2}^{r-1}\sigma_1^{-p_i} \sigma_2^{q_i} \Bigg) \sigma_1^{-p_r} \sigma_2^{q_r-1} .
	\end{align*}
By Lemma~\ref{lem:daltupbound}, we have
\[
\mathrm{dalt}(\widehat{\beta})\leq n.
\]
This completes the proof.
\end{proof}

\section{The alternation number and the dealternating number of a closed positive $3$-braid}
\label{The alternation number and the dealternating number of a closed positive 3-braid}

In this section, we determine the alternation numbers and the dealternating numbers
of some closed positive 3-braid knots.
\begin{theorem} \label{thm:3-braid}
	Let $\beta $ be a positive $3$-braid of the form
	\[
	\Delta^{2n} \prod_{i=1}^{r} \sigma_1^{p_i} \sigma_2^{q_i}
	\]
	such that $\widehat{\beta}$ is a knot and $p_i, q_i\geq 2$ \ for $i=1,2,\dots ,r$,
	where $n\geq 0$, $r\geq 1$. Then we have
	\[
	\mathrm{alt}(\widehat{\beta})=\mathrm{dalt}(\widehat{\beta})=n+r-1.
	\]
\end{theorem}

We recall two $\mathbb{Z}$-valued invariants. 
We denote the \textit{signature}~\cite{mur3} of a link $L$ by $\sigma(L)$
(with the sign conventions such that the right-handed trefoil has signature two).
In this paper, we define the signature of a link diagram as that of the link represented by the diagram.
We denote the \textit{Rasmussen $s$-invariant} of a knot by $s(K)$ \cite{ras1}.
The first author gave a lower bound for the alternation number of a knot.
\begin{lemma}[\cite{abe1}] \label{lem:altlowbound}
	Let $K$ be a knot. Then we have
	\[
	\frac{1}{2} |s(K)-\sigma (K)|\leq \mathrm{alt} (K).
	\]
\end{lemma}

Rasmussen showed the following.
\begin{lemma}[\cite{ras1}] \label{lem:positiverasmussen}
	Let $D$ be a positive diagram of a positive knot $K$. Then we have 
	\[
	s(K)=c(D)-o(D)+1,
	\]
	where $c(D)$ is the crossing number of $D$ and $o(D)$ is the number of Seifert circles of $D$.
\end{lemma}

Erle calculated the signatures of closed 3-braids.
\begin{lemma}[\cite{erl1}] \label{lem:3-braid-signature1}
    Let $\beta $ be a $3$-braid in $\Omega _6$.
	With the notation of Lemma~\ref{lem:3-braidclass}, we have
	\[
	\sigma (\widehat{\beta}) = 4n-\sum_{i=1}^{r}(p_i-q_i).
	\]
\end{lemma}

\begin{lemma} \label{lem:3-braid-signature2}
    With the notation of Theorem~\ref{thm:3-braid}, we have
	\[
	\sigma (\widehat{\beta})=4n-2r+\sum_{i=1}^{r}(p_i+q_i).
	\]
\end{lemma}

\begin{proof}
We modify $\beta $ into $\Delta ^{2n+2r} \prod_{i=1}^{r} \sigma_1^{-1} \sigma_2^{p_i-2} \sigma_1^{-1} \sigma_2^{q_i-2}$
by the following equalities.
\begin{align*}
\beta &=\Delta ^{2n} \prod_{i=1}^{r} \sigma_1^{p_i} \sigma_2^{q_i}\\
      &=\Delta ^{2n} \prod_{i=1}^{r} \sigma_1 \sigma_1^{p_i-1} \sigma_2 \sigma_2^{q_i-1}\\
      &=\Delta ^{2n} \prod_{i=1}^{r} \sigma_1^{p_i-1} \sigma_2 \sigma_2^{q_i-1} \sigma_1\\ 
      &=\Delta ^{2n} \prod_{i=1}^{r} \sigma_2^{-1} \sigma_1^{p_i-1} \sigma_2 \sigma_2^{q_i-1} \sigma_1 \sigma_2\\
      &=\Delta ^{2n} \prod_{i=1}^{r} \sigma_1 \sigma_2^{p_i-1} \sigma_1^{-1} \sigma_2^{q_i-1} \sigma_1 \sigma_2\\
      &=\Delta ^{2n} \prod_{i=1}^{r} \Delta \sigma_1^{-1} \sigma_2^{p_i-2} \sigma_1^{-1} \sigma_2^{q_i-2} \Delta\\
      &=\Delta ^{2n+2r} \prod_{i=1}^{r} \sigma_1^{-1} \sigma_2^{p_i-2} \sigma_1^{-1} \sigma_2^{q_i-2}.
\end{align*}
By Lemma~\ref{lem:3-braid-signature1}, we have
\begin{align*}
\sigma (\widehat{\beta}) &=4(n+r)+\sum_{i=1}^{r}(p_i+q_i-6) \\
						 &=4n-2r+\sum_{i=1}^{r}(p_i+q_i).
\end{align*}
\end{proof}

\begin{proof}[Proof of Theorem~\ref{thm:3-braid}]
First we show that $\mathrm{dalt}(\widehat{\beta})\leq n+r-1$.
We modify $\beta $ into $\sigma_2 (\sigma_1^2 \sigma_2^2)^{n-1} \sigma_1^2 \sigma_2 \sigma_1^{2n}
\prod_{i=1}^{r} \sigma_1^{p_i} \sigma_2^{q_i}$ by the following equalities.
\begin{eqnarray*}
	\beta &=& \Delta ^{2n} \prod_{i=1}^{r} \sigma_1^{p_i} \sigma_2^{q_i}\\
		  &=& (\sigma_2 \sigma_1^2 \sigma_2)^n \sigma_1^{2n} \prod_{i=1}^{r} \sigma_1^{p_i} \sigma_2^{q_i}\\
		  &=& \sigma_2 (\sigma_1^2 \sigma_2^2)^{n-1} \sigma_1^2 \sigma_2 \sigma_1^{2n}
			  \prod_{i=1}^{r} \sigma_1^{p_i} \sigma_2^{q_i}.
\end{eqnarray*}
If $r=1$, we modify $\beta$ into $(\sigma_1^2 \sigma_2^2)^{n-1} \sigma_1^2 \sigma_2 \sigma_1^{2n+p_1} \sigma_2^{q_1+1}$
by the following equalities.
\begin{eqnarray*}
	\beta &=& \sigma_2 (\sigma_1^2 \sigma_2^2)^{n-1} \sigma_1^2 \sigma_2 \sigma_1^{2n} \sigma_1^{p_1} \sigma_2^{q_1}\\
		  &=& (\sigma_1^2 \sigma_2^2)^{n-1} \sigma_1^2 \sigma_2 \sigma_1^{2n+p_1} \sigma_2^{q_1+1}.
\end{eqnarray*}
If $r\geq 2$, we modify $\beta$ into $(\sigma_1^2 \sigma_2^2)^{n-1} \sigma_1^2 \sigma_2 \sigma_1^{2n+p_1} \sigma_2^{q_1} 
\{\prod_{i=2}^{r-1} \sigma_1^{p_i} \sigma_2^{q_i} \} \sigma_1^{p_r} \sigma_2^{q_r+1}$ by the following equalities.
\begin{eqnarray*}
	\beta &=& \sigma_2 (\sigma_1^2 \sigma_2^2)^{n-1} \sigma_1^2 \sigma_2 \sigma_1^{2n} 
			  \prod_{i=1}^{r} \sigma_1^{p_i} \sigma_2^{q_i}\\
		  &=& (\sigma_1^2 \sigma_2^2)^{n-1} \sigma_1^2 \sigma_2 \sigma_1^{2n+p_1} \sigma_2^{q_1} 
			  \{\prod_{i=2}^{r-1} \sigma_1^{p_i} \sigma_2^{q_i} \} \sigma_1^{p_r} \sigma_2^{q_r+1} .
\end{eqnarray*}
By Lemma~\ref{lem:daltupboundpositive3}, we have
\[
\mathrm{dalt}(\widehat{\beta})\leq n+r-1.
\]
Next we show that $\mathrm{alt}(\widehat{\beta})\geq n+r-1.$
Lemmas~\ref{lem:3-braid-signature2} and \ref{lem:positiverasmussen} imply that
\begin{eqnarray*}
s(\widehat{\beta })       &=& 6n-2+\sum_{i=1}^{r}(p_i+q_i),\\
\sigma (\widehat{\beta }) &=& 4n-2r+\sum_{i=1}^{r}(p_i+q_i).
\end{eqnarray*}
By Lemma~\ref{lem:altlowbound}, we have
\[
\frac{1}{2} |s(K)-\sigma (K)|=n+r-1\leq \mathrm{alt}(\widehat{\beta}).
\]
Therefore, we have
\[
\mathrm{alt}(\widehat{\beta})=\mathrm{dalt}(\widehat{\beta})=n+r-1. 
\]
 \hfill{} {\setlength{\fboxsep}{2.5pt}\fbox{}}
\end{proof}
\section{On $\overline{t'_2}$ moves for a closed positive 3-braid}
\label{A relation between the alternateness}

In this section, we show that there exist 
infinitely many positive knots with any dealternating number (or any alternation number) and any braid index.
\begin{theorem} \label{coro:braidaltexist}
	For any $m\in \mathbb{Z}_{\geq 0}$ and $n\geq 3$, there exist infinitely many positive knots $K$ with
	$\mathrm{alt} (K)=\mathrm{dalt} (K)=m$ and $b(K)=n$.
\end{theorem}

\begin{remark} 
If $L$ is a link with $b(L)=2$, then we have $\mathrm{alt} (L)=\mathrm{dalt} (L)=0$ 
since $L$ is a $(2,k)$-torus link for some integer $k\ne 0,\pm 1$.

\end{remark}

A \textit{band move} is a local move along the orientation on a link diagram as shown in Figure~\ref{fig:bandmove}.
\begin{figure}[htbp]
  \begin{center}
    \includegraphics[scale=0.8]{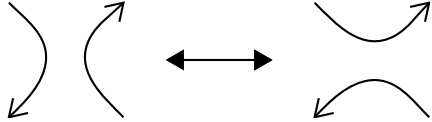}
  \end{center}
  \caption{A band move}
  \label{fig:bandmove}
\end{figure}

\noindent
Murasugi showed the following.
\begin{lemma}[\cite{mur2}] \label{lem:t_2movesignature}
	Let $D$ and $D'$ be link diagrams which are related to each other by one band move. Then we have
	\[
	-1\leq \sigma(D)-\sigma(D')\leq 1.
	\]
\end{lemma}

\begin{remark}[cf.~\cite{kawauchi-1996-book}]
Let $D_+$ and $D_-$ be link diagrams such that $D_-$ is obtained from $D_+$
by a crossing change at a crossing $c$ from positive to negative. 
Let $D/c$ be the diagram obtained from $D_+$ (or $D_-$) by smoothing $c$. 
The following inequalities hold.
\begin{align*}
 0 &\leq  \sigma (D_+) -\sigma (D_-) \leq 2 ,\\ 
-1 &\leq  \sigma (D_{\pm }) -\sigma (D/c) \leq 1 .
\end{align*}
\end{remark}

A \textit{$\overline{t_2'}$ move} is a local move along the orientation on a link diagram 
as shown in Figure~\ref{fig:t_2move}. 
\begin{figure}[htbp]
	\begin{center}
		\includegraphics[scale=0.8]{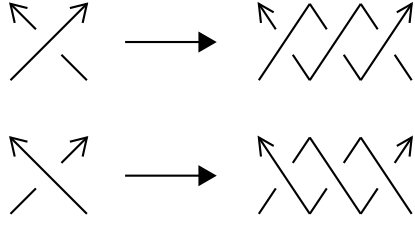}
	\end{center}
	\caption{A $\overline{t_2'}$ move}
	\label{fig:t_2move}
\end{figure}

\begin{lemma} \label{prop:t_2movelinksignature}
	Let $D$ be a link diagram, and let $c$ be a crossing of $D$.
	Let $D^{c(k)}$ be the diagram obtained from $D$ by applying a $\overline{t_2'}$ move $k$ times at $c$.
	Then we have
	\[
	\sigma(D^{c(k)})-\sigma(D) \in \{0, \mathrm{sign}(c), 2\mathrm{sign}(c)\},
	\]
	where $ \mathrm{sign} 
	\left(
    \begin{minipage}{15pt}
        \begin{picture}(15,15)
            \put(0,0){\vector(1,1){15}}
            \qbezier(15,0)(15,0)(10,5)
            \qbezier(5,10)(0,15)(0,15)
            \put(0,15){\vector(-1,1){0}}
        \end{picture}
    \end{minipage}
	\right)=1$ and $\mathrm{sign} \left(
    \begin{minipage}{15pt}
        \begin{picture}(15,15)
            \put(15,0){\vector(-1,1){15}}
            \qbezier(0,0)(0,0)(5,5)
            \qbezier(10,10)(15,15)(15,15)
            \put(15,15){\vector(1,1){0}}
        \end{picture}
    \end{minipage}
	\right)=-1$.
\end{lemma}

\begin{proof}
We may assume that the crossing $c$ is positive, 
since we can prove in a similar way as in the case $c$ is negative.
Let $D/c$ be the diagram obtained from $D$ by smoothing $c$.
By smoothing one crossing of $D^{c(k)}$, we can obtain a diagram which is equivalent to $D/c$.
Then we obtain the inequalities.
\[
-1\leq \sigma(D^{c(k)})-\sigma(D/c) \leq 1 .
\]
Since $D$ is obtained from $D^{c(k)}$ by $k$ crossing changes from positive to negative,
\[
\sigma(D^{c(k)})-\sigma(D)\geq 0.
\]
Then we obtain
\[
0\leq \sigma(D^{c(k)})-\sigma(D) \leq 2.
\]
\end{proof}

\begin{lemma} \label{prop:t_2moveknotsignature}
	Let $D$ be a link diagram, and let $c$ be a crossing of $D$.
	Let $D/c$ be the link diagram obtained from $D$ by smoothing $c$,
	$D^{c(k)}$ the link diagram obtained from $D$ by applying a $\overline{t_2'}$ move $k$ times at $c$.
	If $\sigma(D)-\sigma(D/c)=\mathrm{sign}(c)$, then we have
	\[
	\sigma(D)=\sigma(D^{c(k)}).
	\]
\end{lemma}

\begin{proof}
We may assume that the crossing $c$ is a positive crossing,
since we can prove in a similar way as in the case $c$ is negative. Since we have 
\[
-1\leq \sigma(D^{c(k)})-\sigma(D/c) \leq 1 ,
\]
we have the following inequalities.
\[
-2\leq \sigma(D^{c(k)})-\sigma(D) \leq 0 .
\]
By Proposition~\ref{prop:t_2movelinksignature}, we have 
\[
\sigma(D)=\sigma(D^{c(k)}).
\]
\end{proof}

Yamada showed the following.
\begin{lemma}[\cite{yamada1987-Seifert-circle-braid}] \label{lem:Yamada}
    Let $L$ be a link and $D$ a diagram of $L$. Then we have
	\[
	b(L)\leq o(D).
	\]
\end{lemma}

The \textit{HOMFLY polynomial} $P_L(v,z)$ of a link $L$~\cite{fre1}, \cite{prz1} is
a $2$-variable link polynomial defined by the skein relationship
\[ v^{-1} P \left(
    \begin{minipage}{15pt}
        \begin{picture}(15,15)
            \put(0,0){\vector(1,1){15}}
            \qbezier(15,0)(15,0)(10,5)
            \qbezier(5,10)(0,15)(0,15)
            \put(0,15){\vector(-1,1){0}}
        \end{picture}
    \end{minipage}
;v,z \right) -v P \left(
    \begin{minipage}{15pt}
        \begin{picture}(15,15)
            \put(15,0){\vector(-1,1){15}}
            \qbezier(0,0)(0,0)(5,5)
            \qbezier(10,10)(15,15)(15,15)
            \put(15,15){\vector(1,1){0}}
        \end{picture}
    \end{minipage}
;v,z \right) = z P \left(
    \begin{minipage}{15pt}
        \begin{picture}(15,15)
            \qbezier(0,0)(10,7.5)(0,15)
            \qbezier(15,0)(5,7.5)(15,15)
            \put(15,15){\vector(1,1){0}}
            \put(0,15){\vector(-1,1){0}}
        \end{picture}
    \end{minipage}
;v,z \right) . \]
We denote by $\mathrm{maxdeg}_v P_L$ (resp.~$\mathrm{mindeg}_v P_L$) 
the maximal (resp.~minimal) degree of $P_L(v,z)$ in the variable $v$. 
Franks and Williams, and Morton independently showed the following.
\begin{lemma}[\cite{fra1}, \cite{mor1}] \label{lemma:Franks and William}
	Let $L$ be a link. Then we have
	\[
	\frac{1}{2} \bigl( \mathrm{maxdeg}_v P_L-\mathrm{mindeg}_v P_L \bigr) +1\leq b(L).
	\]
\end{lemma}

For a positive link, Yokota showed the following. 
\begin{lemma}[\cite{yokota-1992-polynomial-inv}] \label{lem:positivehomfly}
	Let $D$ be a positive diagram of a positive link $L$. Then we have
	\[
	\mathrm{mindeg}_v P_L=c(D)-o(D)+1.
	\]
\end{lemma}

Nakamura showed the following.
\begin{lemma} [\cite{nakamura-2004-note-on-braid-index}]\label{lem:positiveMWF}
	Let $L$ be a closed positive $3$-braid. Then we have
	\[
	\frac{1}{2} \bigl( \mathrm{maxdeg}_v P_L-\mathrm{mindeg}_v P_L \bigr) +1= b(L),
	\]
\end{lemma}

\begin{lemma} \label{thm:t_2-move-positive}
	Let $\beta $ be a positive $3$-braid of the form
	\[
	\prod_{i=1}^{r} \sigma_1^{p_i} \sigma_2^{q_i},
	\]
	such that $\widehat{\beta}$ is a knot and $p_i,q_i\geq 3$ for $i=1,2,\dots ,r$. 
	Let $D$ be the closed braid diagram of $\widehat{\beta }$, and let $c$ be a crossing of $D$.
	Let $D^{c(k)}$ be the diagram obtained from $D$ by applying a $\overline{t'_2}$ move $k$ times at $c$,
	and $K'$ the knot represented by $D^{c(k)}$.
	Then we have
	\begin{eqnarray*}
	(1) & & \mathrm{alt}(K')=\mathrm{dalt}(K')=r-1,\\
	(2) & & b(K')=k+3.
	\end{eqnarray*}
\end{lemma}

\begin{proof}
(1) 
We show that $\sigma(D)=\sigma(D^{c(k)})$.
We obtain the $2$-component link diagram $D/c$ by smoothing $c$.
Then $D/c$ is a closed positive $3$-braid diagram of the form
\[
\prod_{i=1}^{r} \sigma_1^{p'_i} \sigma_2^{q'_i},
\]
such that $p'_i,q'_i\geq 2$ for $i=1,2,\dots ,r$. 
By Lemma~\ref{lem:3-braid-signature2}, we have
\[
\sigma(D)-\sigma(D/c)=1.
\]
By using Proposition~\ref{prop:t_2moveknotsignature}, we have $\sigma(D)=\sigma(D^{c(k)})$.
On the other hands, by Lemma~\ref{lem:positiverasmussen}, $\overline{t'_2}$ moves preserve 
the Rasmussen $s$-invariant for a positive knot diagram.
By Lemma~\ref{lem:altlowbound}, we have $\mathrm{alt}(K')\geq r-1$.
We can deform $D^{c(k)}$ into an $(r-1)$-almost alternating diagram as in Figure~\ref{fig:Figure2}. 
Then we have 
\[
\mathrm{alt}(K')=\mathrm{dalt}(K')=r-1.
\]
(2) First we show that $b(K')\leq k+3$. We note that $o(D)=3$ and $o(D^{c(k)})=2k+3$. 
We can deform $D^{c(k)}$ into a knot diagram $D'$ with $o(D')=k+3$ 
by the Reidemeister moves as in Figure~\ref{fig:Figure6}. By Lemma~\ref{lem:Yamada}, 
we have  $b(K')\leq o(D')=k+3$.
\begin{figure}[htbp]
\begin{center}
    \includegraphics[scale=0.55]{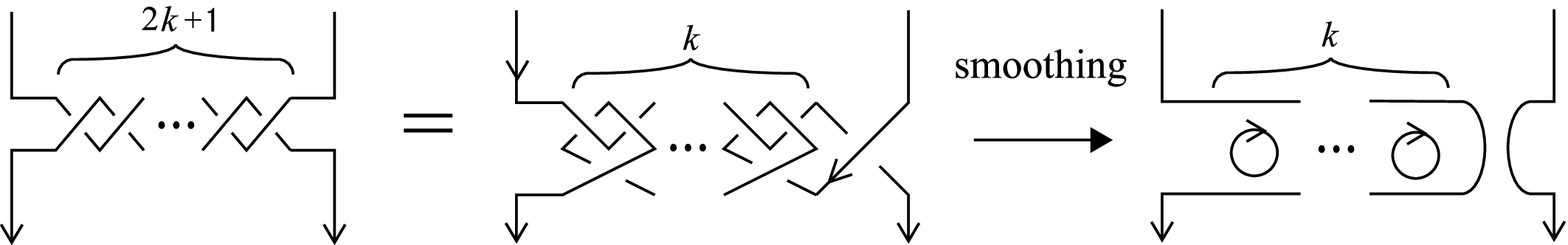}
  \end{center}
  \caption{}
  \label{fig:Figure6}
\end{figure}

Next we show that $b(K')\geq k+3$.
Let $D/c$ be the diagram obtained from $D$ by smoothing $c$.
By using the skein relationship $k$ times, we have the following equation.
\[
P_{D^{c(k)}}(v,z)= v^{2k} P_D(v,z)+v z \Bigg( \sum_{i=1}^{k} v^{2(i-1)} \Bigg) P_{D/c}(v,z).
\]
By Lemma~\ref{lem:positivehomfly}, we have
\[
\mathrm{mindeg}_v P_D = \mathrm{mindeg}_v P_{D/c} +1 =p+q-2.
\]
By Lemma~\ref{lem:positiveMWF}, we have
\[
\mathrm{maxdeg}_v P_D = \mathrm{maxdeg}_v P_{D/c} +1 =p+q+2,
\]
since the braid indices of $D$ and of $D/c$ are equal to three. 
By Lemma~\ref{lemma:Franks and William}, we have
\begin{eqnarray*}
b(K') &\geq & \frac{1}{2} \Bigl( \mathrm{maxdeg}_v P_{D^{c(k)}}-\mathrm{mindeg}_v P_{D^{c(k)}} \Bigr) +1 \\
      & =   & \frac{1}{2} \Big( (p+q+2k+2)-(p+q-2) \Big) +1=k+3.
\end{eqnarray*}
This completes the proof.
\end{proof}

\begin{proof}[Proof of Theorem~\ref{coro:braidaltexist}]
	We give an example of a family of infinitely many positive knots.
	Let $l\in \mathbb{Z}_{\geq 0}$, $m\in \mathbb{Z}_{\geq 0}$ and $n\geq 3$.
	Let $D_i$ ($i=1,2,\dots ,l$) be the closed braid diagram of a positive braid of the form
	\[
	\sigma_1^{2i+1} \sigma_2^3  ( \sigma_1^4 \sigma_2^4 ) ^m .
	\]
	Let $c$ be a crossing of $D_i$ and $D_i^{c(n)}$ the diagram obtained from $D_i$
	by applying a $\overline{t'_2}$ move $n$ times at $c$.
	By Lemma~\ref{thm:t_2-move-positive}, $\mathrm{alt} (K_i^n)=\mathrm{dalt} (K_i^n)=m$ and $b(K_i^n)=n+3$, 
	where $K_i^n$ is a knot which has the diagram $D_i^{c(n)}$.
	On the other hands, by Lemma~\ref{lem:positiverasmussen}, we have $s(K_i^n)=8m+2i+2$.
	Then we have $s(K_i^n)\ne s(K_j^n)$ for $i\ne j$. \\
	This completes the proof.
 \hfill{} {\setlength{\fboxsep}{2.5pt}\fbox{}}
\end{proof}

\section{On the Turaev genus of a non-split link}
\label{Note on the Turaev surface}
In this section, we recall the definition of the Turaev surface and some properties.

An \textit{A-splice} and a \textit{B-splice} are local moves on  a link diagram as in Figure~\ref{fig:splice}.
We denote by $s_A D$ (resp.~$s_B D$) the diagram obtained from $D$ by applying
A-splices (resp.~B-splices) at all crossings of $D$. We denote by $|s_A D|$ (resp.~$|s_B D|$)
the number of the  circles of $s_A D$ (resp.~$s_B D$).
\begin{figure}[htbp]
	  \begin{center}
		\includegraphics[scale=0.8]{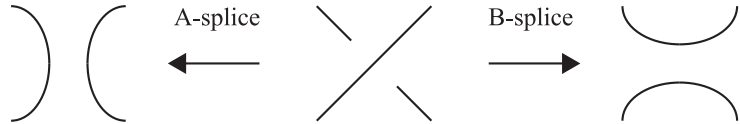}
	  \end{center}
	  \caption{An A-splice and a B-splice}
	  \label{fig:splice}
\end{figure}

We construct the Turaev surface associated to a connected link diagram.	
Let $\Gamma \subset S^{2}$ be the underlying 4-valent graph of a connected link diagram $D$,
and $V$ the union of the vertices of $\Gamma$.
Let $\Gamma \times [-1,1]$ be a surface with
singularities $V \times [-1,1]$ naturally embedded in $S^2 \times  [-1,1]$.
Replace the neighborhoods of $V \times [-1,1]$ with saddle surfaces positioned in such a way that 
the boundary curves in $S^2 \times \{1\}$ (resp.~$S^2 \times \{-1\}$) 
correspond to $s_A D$ (resp.~$s_B D$) as in Figure~\ref{fig:saddle} (see also Figure~\ref{fig:Turaev}).
\begin{figure}[htbp]
	  \begin{center}
		\includegraphics[scale=0.5]{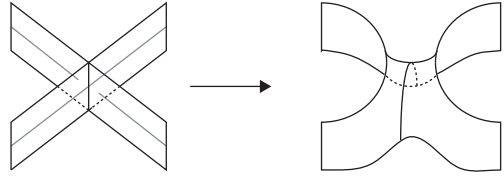}
	  \end{center}
	  \caption{A saddle surface}
	  \label{fig:saddle}
\end{figure}

\begin{figure}[htbp]
	  \begin{center}
		\includegraphics[scale=0.5]{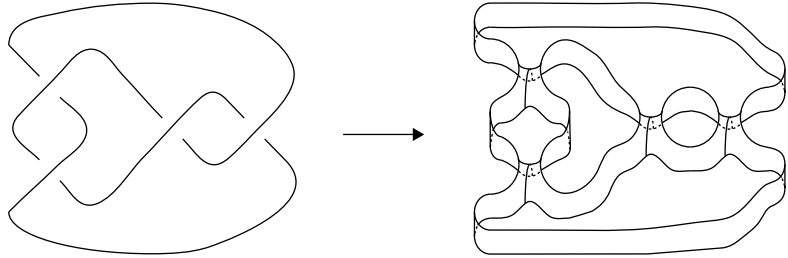}
	  \end{center}
	  \caption{On the construction of the Turaev surface}
	  \label{fig:Turaev}
\end{figure}

The Turaev surface is completed by attaching disjoint discs to the $|s_A D|+|s_B D|$
boundary circles in $S ^3$.
We denote by $g_{T}(D)$ the genus of the Turaev surface associated to $D$.
The \textit{Turaev genus} $g_{T}(L)$ of a non-split link $L$ 
is the minimal number of the genera $g_{T}(D)$ of the Turaev surfaces associated to diagrams $D$ of $L$.  
Turaev showed the following two lemmas.
\begin{lemma}[\cite{cromwell-book}, \cite{turaev-surface-1987}]
Let $D$ be a connected link diagram. Then we have
\[
g_{T}(D)=
\frac{1}{2} (c(D)+2 -|s_A D|-|s_B D|).
\]
\end{lemma}

\begin{lemma} [\cite{turaev-surface-1987}] \label{lem:Turaev-surface-of-alternaitng-links}
	Let $D$ be a connected alternating link diagram. Then we have $g_{T}(D)=0$.
\end{lemma}

Lowrance showed the following.
\begin{lemma}[\cite{lowrance2007kfw}] \label{lem:Turaev-surface}
	Let $D$ and $D'$ be connected link diagrams which are related to each other by a crossing change.
	Then  we have
	\[ 
	|g_{T}(D) - g_{T}(D')| \le 1.
	\] 
\end{lemma}

\begin{corollary}\label{lem:upper-bound}
	Let $L$ be a non-split link. Then we have
	\[
	g_{T}(L) \le \mathrm{dalt} (L).
	\]
\end{corollary}

\begin{proof}
	Set $\mathrm{dalt} (L)=n$. By definition, there exists a knot diagram $D$ of $L$ such that 
	$n$ crossing changes in $D$ turn the diagram into an alternating knot diagram $D'$. 
	By Lemma \ref{lem:Turaev-surface}, we have
	\[ 
	|g_{T}(D) - g_{T}(D')| \le n.
	\] 
	By Lemma \ref{lem:Turaev-surface-of-alternaitng-links}, we obtain $g_{T}(D) \le n$.
\end{proof}

\begin{figure}[htbp]
	  \begin{center}
		\includegraphics[scale=0.6]{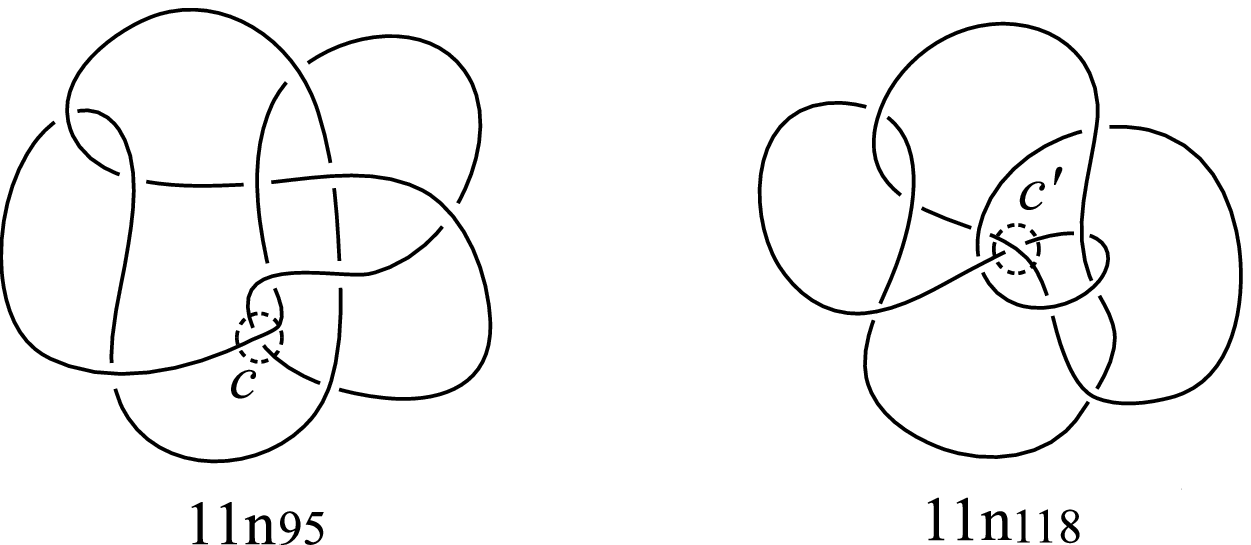}
	  \end{center}
	  \caption{}
	  \label{fig:k11_n95118}
\end{figure}

Since all non-alternating knots of eleven or fewer crossings except $11_{n95}$ and $11_{n118}$
are almost alternating, by Corollary~\ref{lem:upper-bound}, we have the following.
\begin{corollary}
	Let $K$ be a  non-alternating knots of eleven or fewer crossings except $11_{n95}$ and $11_{n118}$.
 	Then we have $g_{T}(K)=1$.  
\end{corollary}

\begin{remark}
	By applying crossing changes at $c$ and $c'$ in Figure~$\ref{fig:k11_n95118}$ (see~\cite{knotinfo}),
    we obtain diagrams of alternating knots. Thus we have
	$ \mathrm{alt} (11_{n95})= \mathrm{alt} (11_{n118})=1$.
	It implies that non-alternating knots of eleven or fewer crossings have alternation number one.
\end{remark}

Khovanov~\cite{khovanov-2000-categorification} introduced an invariant of links, 
called the Khovanov homology, which values bigraded $\mathbb{Z}$-modules and
whose graded Euler characteristic is the Jones polynomial.
We denote by $\omega _{Kh}$(L) the \textit{homological width} (or \textit{thickness})
of the Khovanov homology of a link $L$. Manturov \cite{ole1} and Champanerkar, Kofman and Stoltzfus \cite{cha} gave the following (see also \cite{asa}).
\begin{lemma} [\cite{cha}, \cite{ole1}]\label{thm:Stosic}
	Let $K$ be a knot. Then we have
	\[
	\omega _{Kh}(K)-2 \le g_{T}(K).
	\]
\end{lemma}

Sto$\check {s}$i$\acute {c}$~\cite{sto1} and Turner~\cite{tur1} independently calculated 
the rational Khovanov homology of the $(3, 3n+i)$-torus link, where $i=0,1,2$.
\begin{lemma}[\cite{sto1}, \cite{tur1}]\label{thm:khovanov}
	Let $T_{3,3n+i}$ be the $(3,3n+i)$-torus link, where  $i=0,1,2$.
	Then we have 
	\[
	\omega _{Kh}(T_{3,3n+i})=n+2.
	\]
\end{lemma}

By Theorems~\ref{thm:dalt-upbound-3braid} and~\ref{thm:Stosic}, we obtain the following.
\begin{theorem}\label{thm:the-dealternating-number-of-torus-links}
	Let $T_{3,3n+i}$ be the $(3,3n+i)$-torus knot, where  $i=1,2$.
	Then we have 
	\[g_{T}(T_{3,3n+i})  = \mathrm{dalt}(T_{3,3n+i}) =n.\]
\end{theorem}

\begin{proof}
By Theorem~\ref{thm:dalt-upbound-3braid}, we have $\mathrm{dalt}(T_{3,3n+i}) \le n$.
On the other hand, by Theorems~\ref{thm:Stosic} and~\ref{thm:khovanov}, we have $n \le g_{T}(T_{3,3n+i})$.
\end{proof}


\section{Appendix (authored with In Dae Jong)}
Appendix is authored by the first and the second authors and In Dae Jong. 
In this section, we observe that non-alternating Montesinos links and semi-alternating links
are almost alternating.
Throughout this section, we assume that tangles and tangle diagrams have four ends.
We recall some notations for tangle diagrams. 

Two tangle diagrams $T$ and $T'$ are \textit{equivalent}, denoted by $T \sim T'$,
if they are related by a finite sequence of the Reidemeister moves.
The \textit{integral tangle diagram} $[n]$ ($n$ $\in$ $\mathbb{Z}$) is $n$ horizontal half twists
(see Figure \ref{fig:integraltangles}). 
\begin{figure}[!htbp]
	\begin{center}
	\includegraphics*[scale=0.6]{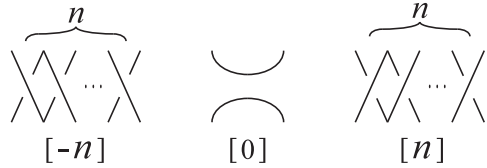}
	\end{center}
	 \caption{The integral tangle diagrams}
	 \label{fig:integraltangles}
\end{figure}

The \textit{sum} (resp.~\textit{product}) of two tangle diagrams $T$ and $S$
is the tangle diagram as in Figure~\ref{fig:sumproduct}.
We denote the sum (resp.~product) of two tangle diagrams $T$ and $S$ by $T+S$ (resp.~$T*S$).
\begin{figure}[htbp]
	\begin{center}
		\includegraphics*[scale=0.6]{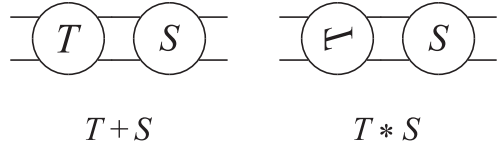}
		\end{center}
	 \caption{The sum and the product of tangle diagrams $T$ and $S$}
	 \label{fig:sumproduct}
\end{figure}	

A \textit{rational tangle diagram} is a tangle diagram obtained from integral tangle diagrams
using only the operation of product.
\begin{remark} \label{remark:rational-tangle}
    Let $a_1, a_2, a_3 ,\dots ,a_n$ be integers. 
    A rational tangle diagram $(\cdots(([a_1]*[a_2])*[a_3])*\cdots )*[a_n]$ is alternating if and only if $a_1, a_2, a_3 ,\dots ,a_n$ have same sign. 
    For any rational tangle diagram $T$, there exists an alternating rational tangle diagram $T'$
	such that $T \sim T'$,
\end{remark}

A \textit{flype} is a local move on a tangle diagram (or a link diagram) as in Figure~\ref{fig:flype}.
\begin{figure}
	\begin{center}
	\includegraphics*[scale=0.7]{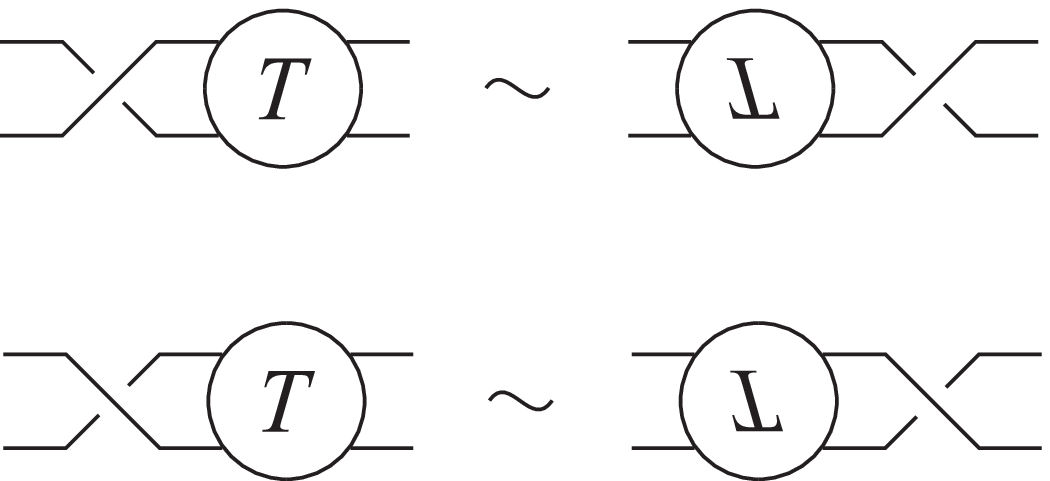}
	\end{center}
	 \caption{A flype}
	 \label{fig:flype}
\end{figure}

\begin{remark}\label{rem:tangle}
	Let $T$ be a rational tangle diagram and $s$ an integer.
	By flypes, we have $T+ [s] \sim [s]+T$ (see~\cite{kauffman-2003-rational-tangle}).
\end{remark}

We depict a symbol $o$ (resp.~$u$) near each end point of a tangle diagram as in Figure~\ref{fig:tangletype}
if an over-crossing (resp.~under-crossing) appears first when we traverse the component from the end point.
We call a tangle diagram on the left side (resp.~right side) 
in Figure~\ref{fig:tangletype} of {\it type 1} (resp.~of {\it type 2}).
We note that an alternating tangle diagram is of type 1 or of type 2.\footnote{Use a checkerboard coloring for the diagram.}
For example, the integral tangle $[1]$ is of type 1, and  $[-1]$ is of type 2. 
We show two lemmas needed later. 	
\begin{figure}[htbp]
	\begin{center}
	\includegraphics*[scale=0.7]{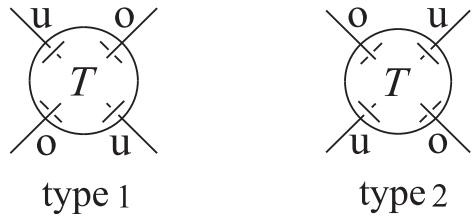}
	\end{center}
	 \caption{Tangles of type $1$ and of type $2$}
	 \label{fig:tangletype}
\end{figure}

\begin{lemma}\label{lem:keylemma1}
	Let $T$ be an alternating tangle diagram, and 
	$R$ an alternating rational tangle diagram.
	Then there exist an alternating tangle $T'$ and an integer $r$ such that $T+R \sim T'+[r]$.
\end{lemma}

\begin{proof}
If $T$ and $R$ are both of type 1 or of type 2, then we set $T'=T+R$ and $r=0$.
We assume that the tangle diagram $T$ is of type 1, and the rational tangle diagram $R$ is of type 2.
In this case, $R$ is represented by $(\cdots([-a_1]*[-a_2])*\cdots )*[-a_n]$,
where $a_1,a_2,\dots ,a_n$ are positive integers.
If $R$ is an integral tangle diagram $[r]$, then we set $T'=T$ and $[r]=R$.
If $R$ is not an integral tangle diagram, 
we obtain the sum of an alternating tangle diagram of type 1 and the integral tangle diagram $[-(a_n+1)]$
by deforming $T+R$ as in Figure~\ref{fig:Figure15}.
If $T$ is of type 2 and $R$ is of type 1, then
the proof is reduced to the case treated above by taking the mirror image of $T+R$.
\begin{figure}[htbp]
	\begin{center}
	\includegraphics*[scale=0.7]{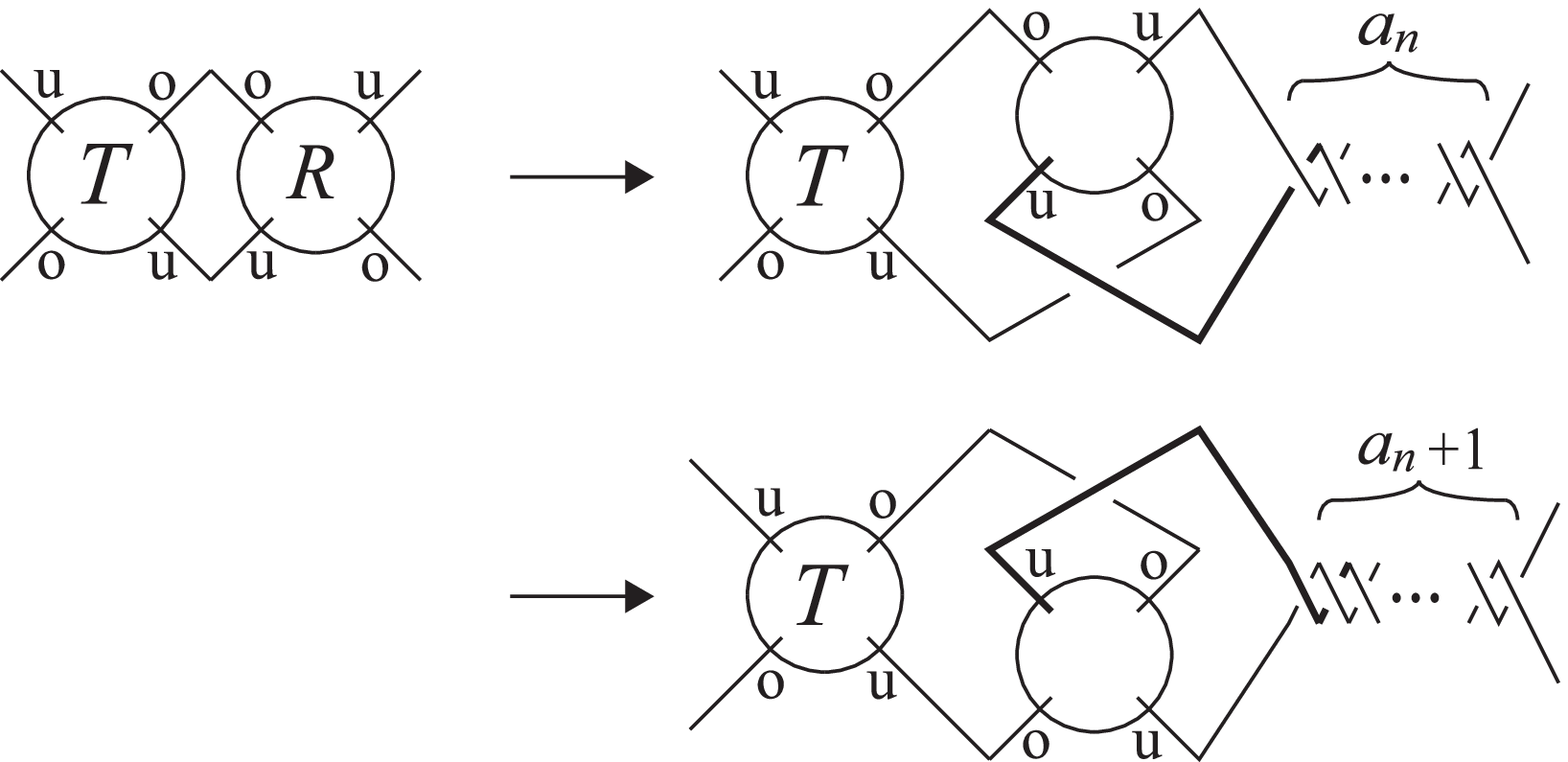}
	\end{center}
	 \caption{}
	 \label{fig:Figure15}
\end{figure}	
\end{proof}

\begin{lemma}\label{lem:keylemma2}
	Let  $R_i$ $(i=1,\dots, m)$ be alternating rational tangle diagrams. 
	Then there exist an alternating tangle diagram $T$
	and an integer $s$ such that $R_1+ \cdots + R_m \sim T+[s]$.
\end{lemma}

\begin{proof}
We prove this lemma by induction on $m$. It is trivial for the case $m=1$.
The lemma is also true by Lemma \ref{lem:keylemma1} for the case $m=2$.

Suppose that the lemma is true for the case $m=n$ such that $n \ge 2$.
We show that the lemma is true for the case $m=n+1$.
By the assumption, there exist an alternating tangle diagram $T$ and $s\in \mathbb{Z}$
such that $R_1+ \cdots+ R_{n} \sim T+[s]$.
By Remark~\ref{rem:tangle}, we have
\[
R_1+ \cdots+ R_{n+1} \sim T+R_{n+1} +[s].
\]				 
By Lemma~\ref{lem:keylemma1}, 
there exist an alternating tangle diagram $T'$ and $r\in \mathbb{Z}$
such that $T+R_{n+1} \sim T'+ [r]$.
We obtain the sum of an alternating tangle diagram $T'$ 
and an integral tangle diagram $[r+s]$. 
\end{proof}

The \textit{numerator} $N(T)$ and the \textit{denominator} $D(T)$ of a tangle diagram $T$
are closures of $T$ as in Figure~\ref{fig:Figure16}.
A \textit{Montesinos link} is a link which has a diagram represented by $N(R_1+ \cdots+ R_m)$,
where $R_1,\dots ,R_m$ are rational tangle diagrams (see Figure~\ref{fig:Montesinoslink}).
\begin{figure}[htbp]
	\begin{center}
	\includegraphics*[scale=0.7]{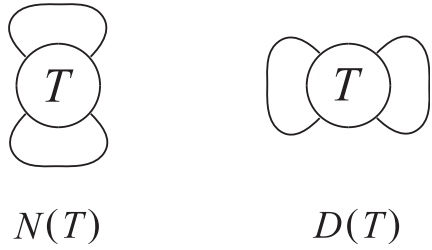}
	\end{center}
    \caption{The numerator and the denominator of a tangle diagram $T$}
    \label{fig:Figure16}
\end{figure}	

\begin{figure}[htbp]
	\begin{center}
	\includegraphics*[scale=0.7]{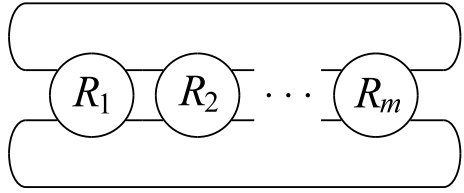}
	\end{center}
	 \caption{A diagram of a Montesinos link}
	 \label{fig:Montesinoslink}
\end{figure}

\begin{proposition} \label{prop:Montesinos}
	Montesinos links are alternating or almost alternating. 
\end{proposition}

\begin{proof}
By Remark~\ref{remark:rational-tangle}, every Montesinos link has 
a diagram $D$ represented by $N(R_1+ \cdots+ R_m)$,
where $R_1,\dots ,R_m$ are alternating rational tangle diagrams.
Then, by Lemma~\ref{lem:keylemma2}, the diagram $D$ can be deformed into a diagram represented by the numerator
$N(T+[s])$ of the sum of an alternating tangle diagram $T$ and an integral tangle diagram $[s]$.

If $T$ and $[s]$ are both of type 1 or of type 2, then $N(T+[s])$ is alternating.
If $T$ is of type 1 and $[s]$ is of type 2,
then we deform $N(T+[s])$  into an almost alternating diagram as in Figure~\ref{fig:Figure18}.
If $T$ is of type 2 and $[s]$ is of type 1, then
the proof is reduced to the case treated above by taking the mirror image of the diagram.

\begin{figure}[htbp]
	\begin{center}
	\includegraphics*[scale=0.7]{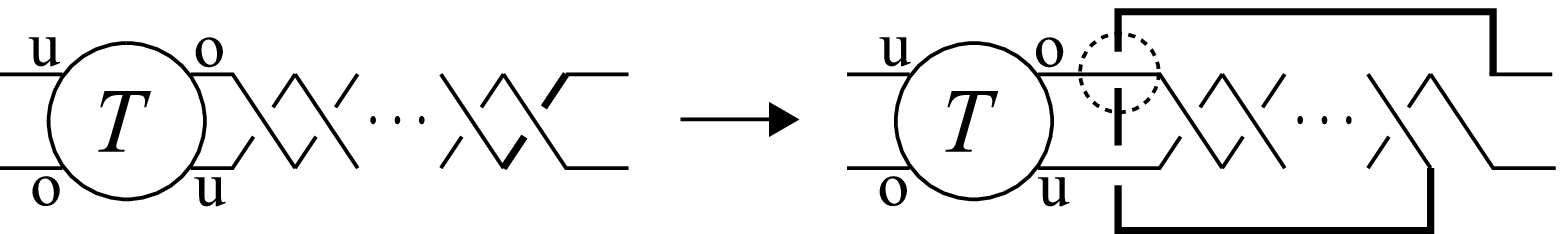}
	\end{center}
	 \caption{}
	 \label{fig:Figure18}
\end{figure}
\end{proof}

\begin{figure}[htbp]
	\begin{center}
	\includegraphics*[scale=0.7]{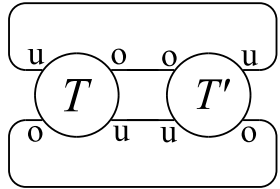}
	\end{center}
	\caption{A diagram of a semi-alternating link}
	\label{fig:semi-alternating}
\end{figure}	

A tangle diagram $T$ is \textit{strongly alternating} if both $N(T)$ and $D(T)$ are reduced alternating.
A link $L$ is \textit{semi-alternating} 
if $L$ has a non-alternating  diagram which is represented by $N(T+T')$,
where $T$ and $T'$ are strongly alternating tangle diagrams (see Figure~\ref{fig:semi-alternating}).

\begin{proposition} \label{thm:semi-alternating}
	Semi-alternating links are almost alternating. 
\end{proposition}

\begin{proof}
Lickorish and Thistlethwaite~\cite{lickolish-thistlethwaite-semi-alternating} showed that semi-alternating links are non-alternating.
On the other hand, we can obtain an almost alternating diagram of a semi-alternating link as in Figure~\ref{fig:Figure20}.
\end{proof}
\begin{figure}[htbp]
	\begin{center}
	\includegraphics*[scale=0.7]{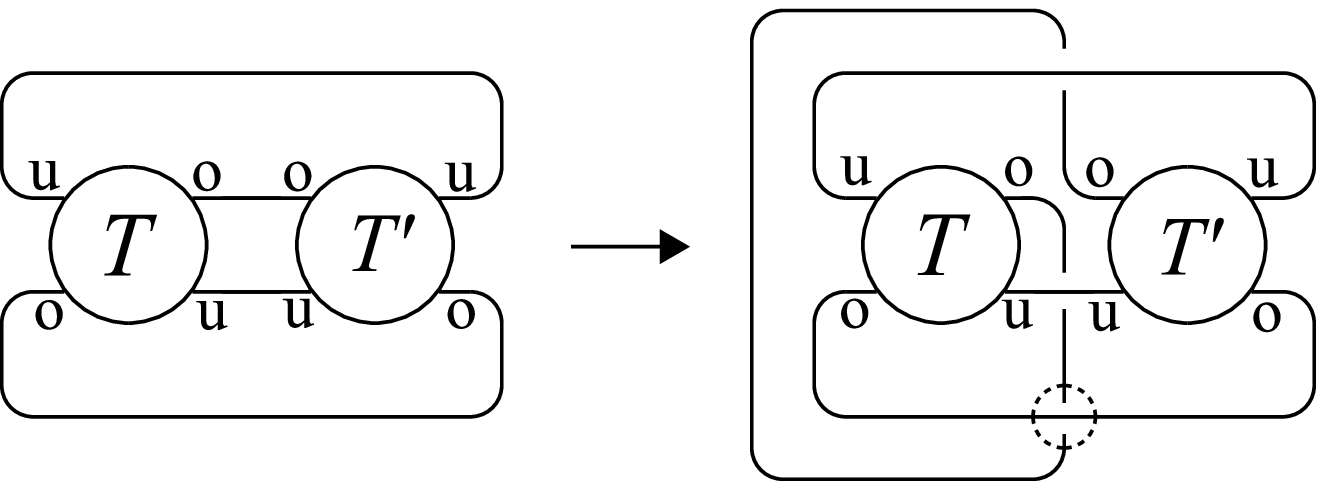}
	\end{center}
	 \caption{}
	 \label{fig:Figure20}
	\end{figure}

For a non-split link $L$,  
Champanerkar and Kofman \cite{cha2} showed 
\[  \omega_{Kh}(L)-2 \le \mathrm{dalt} (L).\] 
Propositions \ref{prop:Montesinos} and \ref{thm:semi-alternating} imply the following.
\begin{corollary}
Let $L$ be a Montesinos link or a semi-alternating link.
Then 
\[ \omega_{Kh}(L) \le 3. \]
\end{corollary}

\section*{Acknowledgments}
The authors would like to express their sincere gratitudes to Akio Kawauchi and Taizo Kanenobu
for helpful advices and comments. 
They also would like to thank the members of Friday Seminar on Knot Theory in Osaka City University
for uncounted discussions, especially Atsushi Ishii and Masahide Iwakiri.


\end{document}